\documentclass{amsproc}
\usepackage{amssymb}
\usepackage{graphicx}
\usepackage{euscript}
\usepackage{setspace}
\usepackage{mathrsfs} 
\usepackage{units}   

\makeatletter
\@namedef{subjclassname@2010}{%
\textup{2010} Mathematics Subject Classification}
\makeatother

\numberwithin{equation}{subsection}

\newtheorem{thm}{Theorem}[subsection]
\newtheorem{cor}[thm]{Corollary}
\newtheorem{lem}[thm]{Lemma}
\newtheorem{pro}[thm]{Proposition}

\newtheorem*{thm*}{Theorem}

\theoremstyle{remark}
\newtheorem{rem}[thm]{Remark}

\theoremstyle{definition}

\DeclareMathOperator{\lin}{\mbox{\sc lin}}
\DeclareMathOperator{\D}{d}
\DeclareMathOperator{\dess}{{\mathsf{Des}}}
\DeclareMathOperator{\dzii}{{\mathsf{Chi}}}

\DeclareMathOperator{\koo}{{\mathsf{root}}}

\DeclareMathOperator{\paa}{{\mathsf{par}}}

\newcommand*{\borel}[1]{{\mathfrak B}(#1)}
\newcommand*{\cbb}{\mathbb C}
\newcommand*{\cc}{\mathcal C}

\newcommand*{\des}[1]{{\dess(#1)}}
\newcommand*{\dz}[1]{{\EuScript D}(#1)}
\newcommand*{\dzi}[1]{\dzii(#1)}
\newcommand*{\dzin}[2]{\dzii^{\langle#1\rangle}(#2)}
\newcommand*{\dzn}[1]{{\EuScript D}^\infty(#1)}
\newcommand*{\ee}{\mathcal E}
\newcommand*{\escr}{{\mathscr{E}_V}}
\newcommand*{\ff}{\mathcal F}
\newcommand*{\Ge}{\geqslant}
\newcommand*{\hh}{\mathcal H}

\newcommand*{\is}[2]{\langle#1,#2\rangle}

\newcommand*{\kk}{\mathcal K}

\newcommand*{\lambdab}{{\boldsymbol\lambda}}
\newcommand*{\lambdabxx}[2]{\lambdab^{\hspace{-.3ex} \langle #1 \rangle}_{u|#2}}
\newcommand*{\lambdai}[1]{{\lambda_{#1}^{\hspace{-.3ex}{\langle} i  \rangle}}}
\newcommand*{\lambdabi}{{\lambdab^{\hspace{-.3ex}\langle i \rangle}}}
\newcommand*{\Le}{\leqslant}
\newcommand*{\mm}{\mathscr M}
\newcommand*{\mui}[1]{{\mu_{#1}^{\hspace{-.25ex}\langle i \rangle}}}
\newcommand*{\nbb}{\mathbb N}
\newcommand*{\ogr}[1]{\boldsymbol B(#1)}

\newcommand*{\pa}[1]{\paa(#1)}
\newcommand*{\quasi}[1]{\mathscr Q(#1)}
\newcommand*{\rbb}{\mathbb R}
\newcommand*{\slam}{S_{\boldsymbol \lambda}}
\newcommand*{\smalloplus}{\raise0pt\hbox{$\scriptscriptstyle \oplus$}}

\newcommand*{\sti}[1]{\mathscr S(#1)}
\newcommand*{\supp}[1]{\mathrm{supp}\,#1}
\newcommand*{\tcal}{{\mathscr T}}

\newcommand*{\varepsiloni}[1]{{\varepsilon_{#1}^{\hspace{-.1ex}\langle i \rangle}}}
\newcommand*{\xx}{\mathcal X}
\newcommand*{\zbb}{\mathbb Z}

\begin{document}
   \title[Unbounded subnormal weighted shifts on directed trees]
{Unbounded subnormal weighted shifts on directed
trees}
   \author[P.\ Budzy\'{n}ski]{Piotr Budzy\'{n}ski}
   \address{Katedra Zastosowa\'{n} Matematyki,
Uniwersytet Rolniczy w Krakowie, ul.\ Balicka 253c,
PL-30198 Krak\'ow}
   \email{piotr.budzynski@ur.krakow.pl}
   \author[Z.\ J.\ Jab{\l}o\'nski]{Zenon Jan Jab{\l}o\'nski}
\address{Instytut Matematyki, Uniwersytet Jagiello\'nski,
ul.\ \L ojasiewicza 6, PL-30348 Kra\-k\'ow, Poland}
   \email{Zenon.Jablonski@im.uj.edu.pl}
   \author[I.\ B.\ Jung]{Il Bong Jung}
   \address{Department of Mathematics, Kyungpook
National University, Daegu 702-701, Korea}
   \email{ibjung@knu.ac.kr}
   \author[J.\ Stochel]{Jan Stochel}
\address{Instytut Matematyki, Uniwersytet Jagiello\'nski,
ul.\ \L ojasiewicza 6, PL-30348 Kra\-k\'ow, Poland}
   \email{Jan.Stochel@im.uj.edu.pl}
   \thanks{Research of the first, second
and fourth authors was supported by the MNiSzW
(Ministry of Science and Higher Education) grant NN201
546438 (2010-2013). The third author was supported by
Basic Science Research Program through the National
Research Foundation of Korea (NRF) grant funded by the
Korea government (MEST) (2009-0093125).}
   \begin{abstract}
A new method of verifying the subnormality of
unbounded Hilbert space operators based on an
approximation technique is proposed. Diverse
sufficient conditions for subnormality of unbounded
weighted shifts on directed trees are established. An
approach to this issue via consistent systems of
probability measures is invented. The role played by
determinate Stieltjes moment sequences is elucidated.
Lambert's characterization of subnormality of bounded
operators is shown to be valid for unbounded weighted
shifts on directed trees that have sufficiently many
quasi-analytic vectors, which is a new phenomenon in
this area.
   \end{abstract}
   \maketitle
   \section{Introduction}
The theory of bounded subnormal operators was
originated by P. Halmos in \cite{hal1}. Nowadays, its
foundations are well-developed (see \cite{con2}; see
also \cite{c-f} for a recent survey article on this
subject). The theory of unbounded symmetric operators
had been established much earlier (see \cite{jvn} and
the monograph \cite{stone}). In view of Naimark's
theorem, these operators are particular cases of {\em
unbounded} subnormal operators, i.e., densely defined
operators having normal extensions in (possibly
larger) Hilbert spaces. The first general results on
unbounded subnormal operators appeared in \cite{bis}
and \cite{foi} (see also \cite{sli}). A systematic
study of this class of operators was undertaken in the
trilogy \cite{StSz3,StSz1,StSz4}. The theory of
unbounded subnormal operators has intimate connections
with other branches of mathematics and quantum physics
(see \cite{sz2,at-cha1,at-cha2} and
\cite{jor,stob,sz1,j-s}). It has been developed in two
main directions, the first is purely theoretical (cf.\
\cite{m-s,StSz2,e-v,dem2,dem3,vas3,al-vas}), the other
is related to special classes of operators (cf.\
\cite{c-j-k,kou,k-t1,k-t2}). In this paper, we will
focus mostly on the class of weighted shifts on
directed trees.

The notion of a weighted shift on a directed tree
generalizes that of a weighted shift on the $\ell^2$
space, the classical object of operator theory (see
\cite{nik,shi,ml}). In a recent paper \cite{j-j-s}
some fundamental properties of weighted shifts on
directed trees have been studied. Although
considerable progress has been made in this field, a
number of important questions have not been answered.
In this paper we continue investigations along these
lines with special emphasis put on the issue of
subnormality of unbounded operators, the case which is
essentially more complicated and not an easy extension
of the bounded one. The main difficulty comes from the
fact that the celebrated Lambert characterization of
subnormality of bounded operators (cf.\ \cite{Lam}) is
no longer valid for unbounded ones (see Section
\ref{subs1}; see also \cite{j-j-s4} for a surprising
counterexample). A new criterion (read:\ sufficient
condition) for subnormality of unbounded operators has
been invented recently in \cite{c-s-sz}. By using it,
we will show that subnormality is preserved by a
certain weak-type limit procedure (see Theorem
\ref{tw1}). This enables us to perform the
approximation process relevant to unbounded weighted
shifts on directed trees. What we get is Theorem
\ref{main}, which is the main result of this paper
(its proof depends essentially on the passage through
weighted shifts that may have zero weights). It
provides a criterion for subnormality of unbounded
weighted shifts on directed trees written in terms of
consistent systems of measures, which is new even in
the case of bounded operators. Roughly speaking, for
bounded and some unbounded operators, the assumption
that $C^\infty$-vectors generates Stieltjes moment
sequences implies subnormality (the reverse
implication is always true, cf.\ Proposition
\ref{necess-gen}). As discussed in Section
\ref{subs1}, there are unbounded formally normal
operators having dense set of $C^\infty$-vectors, for
which this is not true. It is a surprising fact that
there are non-hyponormal operators having dense set of
$C^\infty$-vectors generating Stieltjes moment
sequences. These are carefully constructed weighted
shifts on a leafless directed tree with one branching
vertex (cf.\ \cite{j-j-s4}). They do not satisfy the
consistency condition $2^\circ$ of Lemma
\ref{charsub2} and none of them has consistent system
of measures.

Under some additional assumption, the criterion for
subnormality formulated in Theorem \ref{main} becomes
a full characterization (cf.\ Theorem
\ref{necessdet2}). This is the case in the presence of
quasi-analytic vectors (cf.\ Theorem \ref{main-0}),
which is the first result of this kind (see Section
\ref{cfs} for more comments).
   \section{Preliminaries}
   \subsection{Notation and terminology}
Let $\rbb$ and $\cbb$ stand for the sets of real and
complex numbers respectively. Define
   \begin{align*}
\text{$\zbb_+ = \{0,1,2,3,\ldots\}$, $\nbb =
\{1,2,3,4,\ldots\}$ and $\rbb_+ = \{x \in \rbb \colon x \Ge
0\}$.}
   \end{align*}
We write $\borel{\rbb_+}$ for the $\sigma$-algebra of
all Borel subsets of $\rbb_+$. The closed support of a
positive Borel measure $\mu$ on $\rbb_+$ is denoted by
$\supp \mu$. We write $\delta_0$ for the Borel
probability measure on $\rbb_+$ concentrated at $0$.

Let $A$ be an operator in a complex Hilbert space
$\hh$ (all operators considered in this paper are
linear). Denote by $\dz{A}$ and $A^*$ the domain and
the adjoint of $A$ (in case it exists). Set $\dzn{A} =
\bigcap_{n=0}^\infty\dz{A^n}$; members of $\dzn{A}$
are called {\em $C^\infty$-vectors} of $A$. A linear
subspace $\ee$ of $\dz{A}$ is said to be a {\em core}
of $A$ if the graph of $A$ is contained in the closure
of the graph of the restriction $A|_{\ee}$ of $A$ to
$\ee$. If $A$ is closed, then $\ee$ is a core of $A$
if and only if $A$ coincides with the closure of
$A|_{\ee}$. A closed densely defined operator $N$ in
$\hh$ is said to be {\em normal} if $N^*N=NN^*$
(equivalently:\ $\dz{N}=\dz{N^*}$ and
$\|N^*h\|=\|Nh\|$ for all $h \in \dz{N}$). For other
facts concerning unbounded operators (including normal
ones) that are needed in this paper we refer the
reader to \cite{b-s,weid}. A densely defined operator
$S$ in $\hh$ is said to be {\em subnormal} if there
exists a complex Hilbert space $\kk$ and a normal
operator $N$ in $\kk$ such that $\hh \subseteq \kk$
(isometric embedding) and $Sh = Nh$ for all $h \in \dz
S$. It is clear that subnormal operators are closable
and their closures are subnormal.

In what follows, $\ogr \hh$ stands for the
$C^*$-algebra of all bounded operators $A$ in $\hh$
such that $\dz{A}=\hh$. We write $\lin \ff$ for the
linear span of a subset $\ff$ of $\hh$.
   \subsection{Directed trees}
Let $\tcal=(V,E)$ be a directed tree ($V$ and $E$
stand for the sets of vertices and edges of $\tcal$,
respectively). If $\tcal$ has a root, which will be
denoted by $\koo$, then we write $V^\circ:=V\setminus
\{\koo\}$; otherwise, we put $V^\circ = V$. Set $\dzi
u = \{v\in V\colon (u,v)\in E\}$ for $u \in V$. A
member of $\dzi u$ is called a {\em child} of $u$. For
every vertex $u \in V^\circ$ there exists a unique
vertex, denoted by $\pa u$, such that $(\pa u,u)\in
E$. The correspondence $u \mapsto \pa u$ is a partial
function from $V$ to $V$. For $n \in \nbb$, the
$n$-fold composition of the partial function $\paa$
with itself will be denoted by $\paa^n$. Let $\paa^0$
stand for the identity map on $V$. We call $\tcal$
{\em leafless} if $V = V^\prime$, where $V^\prime:=\{u
\in V \colon \dzi u \neq \varnothing\}$.

It is well-known that (see e.g., \cite[Proposition
2.1.2]{j-j-s}) $\dzi u \cap \dzi v = \varnothing$ for
all $u, v\in V$ such that $u \neq v$, and
   \begin{align} \label{roz}
V^\circ= \bigsqcup_{u\in V} \dzi u.
   \end{align}
(The symbol ``\,$\bigsqcup$\,'' denotes disjoint
union of sets.) For a subset $W \subseteq V$, we
put $\dzi W = \bigsqcup_{v \in W} \dzi v$ and
define $\dzin{0}{W} = W$, $\dzin{n+1}{W} =
\dzi{\dzin{n}{W}}$ for $n\in \zbb_+$ and $\des W
= \bigcup_{n=0}^\infty \dzin n W$. By induction,
we have
   \allowdisplaybreaks
   \begin{align} \label{n+1}
\dzin{n+1}{W} & = \bigcup_{v \in \dzi{W}}
\dzin{n}{\{v\}}, \quad n \in \zbb_+,
   \\
\label{chmn} \dzin{m}{\dzin{n}{W}} & =
\dzin{m+n}{W}, \quad m,n \in \zbb_+.
   \end{align}
We shall abbreviate $\dzin n {\{u\}}$ and
$\des{\{u\}}$ to $\dzin n u$ and $\des{u}$
respectively. We now state some useful properties of
the functions $\dzin{n}{\cdot}$ and $\des{\cdot}$.
   \begin{pro}
If $\tcal$ is a directed tree, then
   \begin{align}  \label{num4}
\dzin{n}{u} &= \{w \in V\colon \paa^n(w)=u\}, \quad n
\in \zbb_+,\, u \in V,
   \\
\label{dzinn2} \dzin{n+1}{u} & = \bigsqcup_{v \in
\dzi{u}} \dzin{n}{v}, \quad n \in \zbb_+,\, u \in V,
   \\
\label{num1} \dzin{n+1}{u} & = \bigsqcup_{v \in
\dzin{n}{u}} \dzi{v}, \quad n \in \zbb_+,\, u \in V,
   \\
\des u & = \bigsqcup_{n=0}^\infty \dzin n u, \quad
u\in V, \label{num3}
   \\
\des{u_1} \cap \des{u_2} & = \varnothing, \quad u_1,
u_2 \in \dzi{u},\, u_1 \neq u_2,\, u \in V.
\label{num3+}
   \end{align}
   \end{pro}
   \begin{proof}
Equality \eqref{num4} follows by induction on $n$.
Combining \eqref{n+1} with the fact that the sets
$\dzin{n} u$, $u \in V$, are pairwise disjoint for
every fixed integer $n \Ge 0$, we get \eqref{dzinn2}.
Equality \eqref{num1} follows from the definition of
$\dzin{n+1}{u}$ and \eqref{roz}. Using the definition
of $\paa$ and the fact that $\tcal$ has no circuits, we
deduce that the sets $\dzin n u$, $n \in \zbb_+$, are
pairwise disjoint. Hence, \eqref{num3} holds. Assertion
\eqref{num3+} can be deduced from \eqref{num4} and
\eqref{num3}.
   \end{proof}
   \subsection{Weighted shifts on directed trees}
In what follows, given a directed tree $\tcal$, we
tacitly assume that $V$ and $E$ stand for the sets of
vertices and edges of $\tcal$ respectively. Denote by
$\ell^2(V)$ the Hilbert space of all square summable
complex functions on $V$ with the standard inner
product $\is fg = \sum_{u \in V} f(u)
\overline{g(u)}$. For $u \in V$, we define $e_u \in
\ell^2(V)$ to be the characteristic function of the
one-point set $\{u\}$. Then $\{e_u\}_{u\in V}$ is an
orthonormal basis of $\ell^2(V)$. Set $\escr = \lin
\{e_u\colon u \in V\}$.

For $\lambdab = \{\lambda_v\}_{v \in V^\circ}
\subseteq \cbb$, we define the operator $\slam$ in
$\ell^2(V)$ by
   \begin{align*}
   \begin{aligned}
\dz {\slam} & = \{f \in \ell^2(V) \colon
\varLambda_\tcal f \in \ell^2(V)\},
   \\
\slam f & = \varLambda_\tcal f, \quad f \in \dz
{\slam},
   \end{aligned}
   \end{align*}
where $\varLambda_\tcal$ is the mapping defined on
functions $f\colon V \to \cbb$ via
   \begin{align} \label{lamtauf}
(\varLambda_\tcal f) (v) =
   \begin{cases}
\lambda_v \cdot f\big(\pa v\big) & \text{ if } v\in
V^\circ,
   \\
0 & \text{ if } v=\koo.
   \end{cases}
   \end{align}
We call $\slam$ a {\em weighted shift} on the directed
tree $\tcal$ with weights $\lambdab=\{\lambda_v\}_{v
\in V^\circ}$.

Given a family $\{\lambda_v\}_{v \in V^\circ}
\subseteq \cbb$, we define the family
$\{\lambda_{u\mid v}\}_{u \in V, v \in \des{u}}$ by
   \begin{align}    \label{luv}
\lambda_{u\mid v} =
   \begin{cases}
1 & \text{ if } v=u,
   \\
\prod_{j=0}^{n-1} \lambda_{\paa^{j}(v)} & \text{ if }
v \in \dzin{n}{u}, \, n \Ge 1.
   \end{cases}
   \end{align}
Note that due to \eqref{num3} the above definition is
correct and \allowdisplaybreaks
   \begin{align} \label{num2}
\lambda_{u\mid w} & = \lambda_{u\mid v}\lambda_w,
\quad w \in \dzi{v}, \, v\in \des{u}, \, u \in V,
   \\
\lambda_{\pa v\mid w} & = \lambda_v \lambda_{v\mid w},
\quad v \in V^\circ, \, w\in \des v. \label{recfor2}
   \end{align}
The following lemma is a generalization of \cite[Lemma
6.1.1]{j-j-s} to the case of unbounded operators. From
now on, we adopt the convention that
$\sum_{v\in\varnothing} x_v=0$.
   \begin{lem} \label{lem4}
Let $\slam $ be a weighted shift on a directed tree
$\tcal$ with weights $\lambdab = \{\lambda_v\}_{v \in
V^\circ}$. Fix $u \in V$ and $n \in \zbb_+$. Then the
following assertions hold\/{\em :}
   \begin{enumerate}
   \item[(i)] $e_u \in \dz{\slam^n}$ if and only if
$\sum_{v \in \dzin{m}{u}} |\lambda_{u\mid v}|^2 <
\infty$ for all integers $m$ such that $1 \Le m \Le
n$,
   \item[(ii)] if $e_u \in \dz{\slam^n}$, then
$\slam^n e_u = \sum_{v \in \dzin{n}{u}} \lambda_{u\mid
v} \, e_v$,
   \item[(iii)] if $e_u \in \dz{\slam^n}$, then
$\|\slam^n e_u\|^2 = \sum_{v \in \dzin{n}{u}}
|\lambda_{u\mid v}|^2$.
   \end{enumerate}
   \end{lem}
   \begin{proof} For $k \in \zbb_+$,
we define the complex function $\lambdabxx{k}{\cdot}$
on $V$ by
   \begin{align}  \label{deflam}
\lambdabxx{k}{v} =
   \begin{cases}
\lambda_{u|v} & \text{ if } v \in \dzin{k}{u},
   \\
0 & \text{ if } v \in V \setminus \dzin{k}{u}.
   \end{cases}
   \end{align}
We shall prove that for every $k\in \zbb_+$, the
following two conditions hold
   \begin{gather}  \label{small}
e_u \in \dz{\slam^k} \text{ if and only if } \sum_{v
\in \dzin{m}{u}} |\lambda_{u\mid v}|^2 < \infty \text{
for } m = 0,1, \ldots, k,
   \\
   \label{small2} \text{if } e_u \in \dz{\slam^k},
\text{ then } \slam^k e_u = \lambdabxx{k}{\cdot}.
   \end{gather}
We use an induction on $k$. The case of $k=0$ is
obvious. Suppose that \eqref{small} and \eqref{small2}
hold for all nonnegative integers less than or equal
to $k$. Assume that $e_u \in \dz{\slam^k}$. Now we
compute $\varLambda_\tcal (\slam^k e_u)$. It follows
from the induction hypothesis and \eqref{deflam} that
   \allowdisplaybreaks
   \begin{align*}
(\varLambda_\tcal (\slam^k e_u))(v)
&\overset{\eqref{lamtauf}}=
   \begin{cases}
\lambda_v (\slam^k e_u)(\paa(v)) & \text{ if } v \in
V^\circ,
   \\
0 & \text{ if } v=\koo,
   \end{cases}
   \\
&\overset{\eqref{small2}}=
   \begin{cases}
\lambda_v \lambdabxx{k}{\paa(v)} & \text{ if } \paa(v)
\in \dzin{k}{u},
   \\
0 & \text{ otherwise,}
   \end{cases}
   \\
& \overset{\eqref{num4}}= \begin{cases} \lambda_v
\lambda_{u|\paa(v)} & \text{ if } v \in \dzin{k+1}{u},
   \\
0 & \text{ otherwise,}
   \end{cases}
   \\
& \overset{\eqref{num2}}= \lambdabxx{k+1}{v}, \quad v
\in V,
  \end{align*}
which shows that $\varLambda_\tcal (\slam^k e_u) =
\lambdabxx{k+1}{\cdot}$. This in turn implies that
\eqref{small} and \eqref{small2} hold for $k+1$ in
place of $k$. This proves (i) and (ii). Assertion
(iii) is a direct consequence of (ii).
   \end{proof}
The following result is an essential ingredient of the
proof of Theorem \ref{main}.
   \begin{pro}\label{potegi}
If $\lambdabi =\big\{\lambdai{v}\big\}_{v \in
V^\circ}$, $i=1,2,3, \ldots$, and
$\lambdab=\{\lambda_v\}_{v\in V^\circ}$ are families
of complex numbers such that
   \begin{enumerate}
   \item[(i)] $\escr \subseteq \dzn{\slam}
\cap \bigcap_{i=1}^\infty \dzn{S_{\lambdabi}}$,
   \item[(ii)] $\lim_{i \to \infty} \lambdai{v} =
\lambda_v$ for all $v \in V^\circ$,
   \item[(iii)] $\lim_{i \to \infty}
\|S_{\lambdabi}^n e_u\| = \|\slam^n e_u\|$ for all $n
\in \zbb_+$ and $u \in V$,
   \end{enumerate}
then
   \begin{align}  \label{slim+}
\is{\slam^m e_u}{\slam^n e_v} = \lim_{i \to \infty}
\is{S_{\lambdabi}^m e_u}{S_{\lambdabi}^n e_v}, \quad
u,v \in V, \, m,n \in \zbb_+.
   \end{align}
   \end{pro}
   \begin{proof} We split the proof into two steps.

{\sc Step 1.} If $\lambdab=\{\lambda_v\}_{v\in
V^\circ}$ is a family of complex numbers such that
$\escr \subseteq \dzn{\slam}$, then for all $m,n \in
\zbb_+$ and $u,v \in V$,
   \begin{align} \label{smsn}
\is{\slam^m e_u}{\slam^n e_v} =
   \begin{cases}
0 & \text{ if } \cc^{m,n}(u,v) = \varnothing,
   \\
\overline{\lambda_{v|u}} \, \|\slam^m e_u\|^2 & \text{
if } \cc^{m,n}(u,v) \neq \varnothing \text{ and } m\Le
n,
   \\
\lambda_{u|v} \, \|\slam^n e_v\|^2 & \text{ if }
\cc^{m,n}(u,v) \neq \varnothing \text{ and } m >n,
   \end{cases}
   \end{align}
where $\cc^{m,n}(u,v) := \dzin{m}{u} \cap
\dzin{n}{v}$.

Indeed, it follows from Lemma \ref{lem4} that
   \begin{align}
   \begin{aligned}  \label{slmsln}
\is{\slam^m e_u}{\slam^n e_v} & =
\Big\langle\sum_{u^\prime \in \dzin{m}{u}}
\lambda_{u\mid u^\prime} \, e_{u^\prime},
\sum_{v^\prime \in \dzin{n}{v}} \lambda_{v\mid
v^\prime} \, e_{v^\prime}\Big\rangle
   \\
& = \sum_{u^\prime \in \cc^{m,n}(u,v)} \lambda_{u\mid
u^\prime} \overline{\lambda_{v\mid u^\prime}}.
   \end{aligned}
   \end{align}
Hence, if $\cc^{m,n}(u,v) = \varnothing$, then the
left-hand side of \eqref{smsn} is equal to $0$ as
required. Suppose now that $\cc^{m,n}(u,v) \neq
\varnothing$ and $m \Le n$. Then
   \begin{align} \label{num8}
\cc^{m,n}(u,v)=\dzin{m}{u}.
   \end{align}
To show this, take $w \in \cc^{m,n}(u,v)$. Then, by
\eqref{num4}, $u=\paa^m(w)$ and
   \begin{align*}
v = \paa^{n}(w) = \paa^{n-m}(\paa^m(w)) =
\paa^{n-m}(u),
   \end{align*}
which, by \eqref{num4} again, is equivalent to
   \begin{align} \label{num7}
u \in \dzin{n-m}{v}.
   \end{align}
This implies that
   \begin{align} \label{num6}
\dzin{m}{u} \subseteq \dzin{m}{\dzin{n-m}{v}}
\overset{\eqref{chmn}}{=} \dzin{n}{v}.
   \end{align}
Thus \eqref{num8} holds. Next, we show that
   \begin{align} \label{num5}
\lambda_{v\mid u^\prime} = \lambda_{u\mid u^\prime}
\lambda_{v|u}, \quad u^\prime \in \dzin{m}{u}.
   \end{align}
It is enough to consider the case where $m\Ge 1$ and
$n > m$. Since $u^\prime \in \dzin{m}{u}$, we infer
from \eqref{num6} that $u^\prime \in \dzin{n}{v}$.
Moreover, by \eqref{num7}, $u \in \dzin{n-m}{v}$. All
these facts together with \eqref{luv} imply that
   \begin{multline*}
\lambda_{v|u^\prime} = \prod_{j=0}^{n-1}
\lambda_{\paa^{j}(u^\prime)} = \prod_{j=0}^{m-1}
\lambda_{\paa^{j}(u^\prime)} \prod_{j=m}^{n-1}
\lambda_{\paa^{j}(u^\prime)}
   \\
\overset{\eqref{luv}}= \lambda_{u\mid u^\prime}
\prod_{j=0}^{n-m-1}
\lambda_{\paa^{j}(\paa^{m}(u^\prime))}
\overset{\eqref{num4}}= \lambda_{u\mid u^\prime}
\prod_{j=0}^{n-m-1} \lambda_{\paa^{j}(u)}
\overset{\eqref{luv}}= \lambda_{u\mid u^\prime}
\lambda_{v|u},
   \end{multline*}
which completes the proof of \eqref{num5}. Now
applying \eqref{slmsln}, \eqref{num8}, \eqref{num5}
and Lemma \ref{lem4}\,(iii), we obtain
\allowdisplaybreaks
   \begin{align*}
\is{\slam^m e_u}{\slam^n e_v} & = \sum_{u^\prime \in
\dzin{m}{u}} \lambda_{u\mid u^\prime}
\overline{\lambda_{v\mid u^\prime}}
   \\
& \hspace{-2.2ex} \overset{\eqref{num5}}=
\overline{\lambda_{v|u}} \sum_{u^\prime \in
\dzin{m}{u}} |\lambda_{u\mid u^\prime}|^2 =
\overline{\lambda_{v|u}} \, \|\slam^m e_u\|^2.
   \end{align*}
Taking the complex conjugate and making appropriate
substitutions, we infer from the above that
$\is{\slam^m e_u}{\slam^n e_v} = \lambda_{u|v} \,
\|\slam^n e_v\|^2$ if $\cc^{m,n}(u,v) \neq
\varnothing$ and $m
>n$, which completes the proof of Step 1.

{\sc Step 2.} Under the assumptions of Proposition
\ref{potegi}, equality \eqref{slim+} holds.

Indeed, it follows from (ii) that
   \begin{align} \label{wzj+}
\lim_{i \to \infty} \lambdai{u\mid v} = \lambda_{u\mid
v}, \quad u \in V, v \in \des{u},
   \end{align}
where $\{\lambdai{u\mid v}\}_{u \in V, v \in \des{u}}$
is the family related to $\big\{\lambdai{v}\big\}_{v
\in V^\circ}$ via \eqref{luv}. Now, applying Step 1 to
the operators $S_{\lambdabi}$ and $\slam$ (which is
possible due to (i)) and using \eqref{wzj+} and (iii),
we obtain \eqref{slim+}.
   \end{proof}
   \subsection{Backward extensions of Stieltjes moment sequences}
We say that a sequence $\{t_n\}_{n=0}^\infty$ of real
numbers is a {\em Stieltjes moment sequence} if there
exists a positive Borel measure $\mu$ on $\rbb_+$ such
that
   \begin{align*}
t_{n}=\int_0^\infty s^n \D\mu(s),\quad n\in \zbb_+,
   \end{align*}
where $\int_0^\infty$ means integration over the set
$\rbb_+$; $\mu$ is called a {\em representing measure}
of $\{t_n\}_{n=0}^\infty$. A Stieltjes moment sequence
is said to be {\em determinate} if it has only one
representing measure. By the Stieltjes theorem (cf.\
\cite[Theorem~ 1.3]{sh-tam} or \cite[Theorem
6.2.5]{ber}), a sequence $\{t_n\}_{n=0}^\infty
\subseteq \rbb$ is a Stieltjes moment sequence if and
only if the sequences $\{t_n\}_{n=0}^\infty$ and
$\{t_{n+1}\}_{n=0}^\infty$ are positive definite
(recall that a sequence $\{t_n\}_{n=0}^\infty
\subseteq \rbb$ is said to be {\em positive definite}
if $\sum_{k,l=0}^n t_{k+l} \alpha_k
\overline{\alpha_l} \Ge 0$ for all $\alpha_0,\ldots,
\alpha_n \in \cbb$ and $n \in \zbb_+$). It is clear
from the definition that
   \begin{align}  \label{st+1}
\text{if $\{t_n\}_{n=0}^\infty$ is a Stieltjes moment
sequence, then so is $\{t_{n+1}\}_{n=0}^\infty$.}
   \end{align}
The converse is not true in general. Moreover, if
$\{t_n\}_{n=0}^\infty$ is an indeterminate Stieltjes
moment sequence, then so is $\{t_{n+1}\}_{n=0}^\infty$
(see Lemma \ref{bext} or \cite[Proposition
5.12]{sim}). The converse implication fails to hold
(cf.\ \cite[Corollary 4.21]{sim}; see also
\cite{j-j-s4}).

The question of backward extendibility of Hamburger
moment sequences has well-known solutions (see e.g.,
\cite{wri} and \cite{sz}). Below, we formulate a
solution of a variant of this question for Stieltjes
moment sequences (see \cite[Lemma 6.1.2]{j-j-s} for
the special case of compactly supported representing
measures).
   \begin{lem} \label{bext}
Let $\{t_n\}_{n=0}^\infty$ be a Stieltjes moment
sequence and let $\vartheta$ be a positive real
number. Set $t_{-1}=\vartheta$. Then the following are
equivalent{\em :}
   \begin{enumerate}
   \item[(i)] $\{t_{n-1}\}_{n=0}^\infty$ is a
Stieltjes moment sequence,
   \item[(ii)]  $\{t_{n-1}\}_{n=0}^\infty$ is positive
definite,
   \item[(iii)] there is a representing measure $\mu$
of $\{t_n\}_{n=0}^\infty$ such
that\/\footnote{\;\label{foot}We adhere to the
convention that $\frac 1 0 := \infty$. Hence,
$\int_0^\infty \frac 1 s \D \mu(s) < \infty$ implies
$\mu(\{0\})=0$.} $\int_0^\infty \frac 1 s \D \mu(s)
\Le \vartheta$.
   \end{enumerate}
Moreover, if {\em (i)} holds, then the mapping
$\mm_0(\vartheta) \ni \mu \to \nu_{\mu} \in
\mm_{-1}(\vartheta)$ defined by
   \begin{align}   \label{nu}
\nu_{\mu}(\sigma) = \int_\sigma \frac 1 s \D \mu(s) +
\Big(\vartheta - \int_0^\infty \frac 1 s \D
\mu(s)\Big) \delta_0(\sigma), \quad \sigma \in
\borel{\rbb_+},
   \end{align}
is a bijection with the inverse $ \mm_{-1}(\vartheta)
\ni \nu \to \mu_{\nu} \in\mm_0(\vartheta)$ given by
   \begin{align} \label{mu}
\mu_{\nu} ( \sigma) = \int_\sigma s \D \nu (s), \quad
\sigma \in \borel{\rbb_+},
   \end{align}
where $\mm_0(\vartheta)$ stands for the set of all
representing measures $\mu$ of $\{t_n\}_{n=0}^\infty$
such that $\int_0^\infty \frac 1 s \D \mu(s) \Le
\vartheta$, and $\mm_{-1}(\vartheta)$ for the set of
all representing measures $\nu$ of
$\{t_{n-1}\}_{n=0}^\infty$. In particular,
$\nu_{\mu}(\{0\})=0$ if and only if $\int_0^\infty
\frac 1 s \D \mu(s)=\vartheta$.

If {\em (i)} holds and $\{t_n\}_{n=0}^\infty$ is
determinate, then $\{t_{n-1}\}_{n=0}^\infty$ is
determinate, the unique representing measure $\mu$ of
$\{t_n\}_{n=0}^\infty$ satisfies the inequality
$\int_0^\infty \frac 1 s \D \mu(s) \Le \vartheta$, and
$\nu_{\mu}$ is the unique representing measure of
$\{t_{n-1}\}_{n=0}^\infty$.
   \end{lem}
   \begin{proof}
Equivalence (i)$\Leftrightarrow$(ii) follows from the
Stieltjes theorem.

(iii)$\Rightarrow$(i) Clearly, if $\mu \in
\mm_0(\vartheta)$, then $t_{n-1}= \int_0^\infty s^n \D
\nu_{\mu}(s)$ for all $n \in \zbb_+$, which means that
$\{t_{n-1}\}_{n=0}^\infty$ is a Stieltjes moment sequence
and $\nu_{\mu} \in \mm_{-1}(\vartheta)$.

(i)$\Rightarrow$(iii) Take $\nu \in \mm_{-1}(\vartheta)$.
Setting $\mu:=\mu_{\nu}$ (cf.\ \eqref{mu}), we see that
   \begin{align}  \label{tnrep}
t_n = t_{(n+1)-1} = \int_0^\infty s^n s\D \nu(s) =
\int_0^\infty s^n \D \mu(s), \quad n \in \zbb_+.
   \end{align}
It is clear that $\mu(\{0\})=0$ and thus
   \begin{align*}
\int_0^\infty \frac 1 s \D \mu(s) & =
\int_{(0,\infty)} \D \nu(s) = \nu((0,\infty))
   \\
   & = \int_{[0,\infty)} s^0 \D \nu(s) - \nu(\{0\}) =
\vartheta - \nu(\{0\}),
   \end{align*}
which implies that $\int_0^\infty \frac 1 s \D \mu(s)
\Le \vartheta$. This, combined with \eqref{tnrep},
shows that $\mu \in \mm_0(\vartheta)$. Since
$\nu(\rbb_+)=\vartheta$, we deduce from \eqref{nu} and
the definition of $\mu$ that
   \allowdisplaybreaks
   \begin{align*}
\nu_{\mu}(\sigma) &= \int_{\sigma\setminus \{0\}}
\frac 1 s \D \mu(s) + \Big(\vartheta-\int_0^\infty
\frac 1 s \D \mu(s)\Big) \delta_0(\sigma \cap \{0\})
   \\
&= \nu(\sigma\setminus \{0\}) + \Big(\vartheta-
\nu((0, \infty))\Big) \delta_0(\sigma \cap \{0\})
   \\
&= \nu(\sigma\setminus \{0\}) + \nu(\{0\})
\delta_0(\sigma \cap \{0\}) = \nu(\sigma), \quad
\sigma \in \borel{\rbb_+},
   \end{align*}
which yields $\nu_{\mu} = \nu$.

We have proved that, under the assumption (i), the
mapping $\mm_0(\vartheta) \ni \mu \to \nu_{\mu} \in
\mm_{-1}(\vartheta)$ is well-defined and surjective.
Its injectivity follows from the equality
   \begin{align*}
\mu(\sigma) = \mu(\sigma \setminus \{0\}) =
\int_{\sigma \setminus \{0\}} s \D \nu_{\mu}(s), \quad
\sigma \in \borel{\rbb_+}, \, \mu \in \mm_0(\vartheta).
   \end{align*}
This yields the determinacy part of the conclusion.
   \end{proof}
   \begin{rem}
Suppose that $\{t_n\}_{n=0}^\infty$ is a determinate
Stieltjes moment sequence with a representing measure
$\mu$. If $\int_0^\infty \frac 1 s \D \mu(s) =
\infty$, then $(\vartheta, t_0, t_1, \ldots)$ is never
a Stieltjes moment sequence. In turn, if
$\int_0^\infty \frac 1 s \D \mu(s) < \infty$, then
$(\vartheta, t_0, t_1, \ldots)$ is a determinate
Stieltjes moment sequence if $\vartheta \Ge
\int_0^\infty \frac 1 s \D \mu(s)$, and it is not a
Stieltjes moment sequence if $\vartheta <
\int_0^\infty \frac 1 s \D \mu(s)$.
   \end{rem}
   \begin{rem}
Under the assumptions of Lemma \ref{bext}, if
$\{t_{n-1}\}_{n=0}^\infty$ is a Stieltjes moment
sequence and $t_0 > 0$, then $t_n > 0$ for all $n
\in \zbb_+$ and
   \begin{align*}
\sup_{n \in \zbb_+}\frac{t_n^2}{t_{2n+1}} \Le
\int_0^\infty \frac{1}{s} \D \mu(s) \Le \vartheta,
\quad \mu \in \mm_0(\vartheta).
   \end{align*}
Indeed, since $t_0 > 0$ and $\mu(\{0\})=0$, we see
that $t_n > 0$ for all $n \in \zbb_+$. Thus
   \begin{align*}
t_n^2 = \Big(\int_{(0,\infty)}
s^{-\nicefrac12}s^{n+\nicefrac12} \D \mu(s)\Big)^2 \Le
\int_0^\infty \frac{1}{s} \D\mu(s) \int_0^\infty
s^{2n+1} \D\mu(s), \quad n \in \zbb_+.
   \end{align*}
Note that if $\{t_n\}_{n=0}^\infty$ is
indeterminate, then there is a smallest
$\vartheta$ for which the sequence
$\{t_{n-1}\}_{n=0}^\infty$ is a Stieltjes moment
sequence (see \cite{j-j-s4} for more details).
   \end{rem}
   \section{A General Setting for Subnormality}
   \subsection{Criteria for subnormality}
The only known general characterization of
subnormality of unbounded Hilbert space operators is
due to Bishop and Foia\c{s} (cf.\ \cite{bis,foi}; see
also \cite{FHSz} for a new approach via sesquilinear
selection of elementary spectral measures). Since this
characterization refers to semispectral measures (or
elementary spectral measures), it seems to be useless
in the context of weighted shifts on directed trees.
The other known criteria for subnormality require the
operator in question to have an invariant domain (with
the exception of \cite{sz4}). Since a closed subnormal
operator with an invariant domain is automatically
bounded (cf.\ \cite[Theorem 3.3]{ota}) and a weighted
shift operator $\slam$ on a directed tree is always
closed (cf.\ \cite[Proposition 3.1.2]{j-j-s}), we have
to find a smaller subspace of $\dz{\slam}$ which is an
invariant core of $\slam$. This will enable us to
apply the aforesaid criteria for subnormality of
operators with invariant domains in the context of
weighted shift operators on directed trees.

Using a recent result from \cite{c-s-sz}, we obtain
the following criterion for subnormality which is a
key tool for proving Theorem \ref{main}.
   \begin{thm} \label{tw1}
Let $\{S_{\omega}\}_{\omega \in \varOmega}$ be a net
of subnormal operators in a complex Hilbert space
$\hh$ and let $S$ be a densely defined operator in
$\hh$. Suppose that there is a subset $\xx$ of $\hh$
such that
   \begin{enumerate}
   \item[(i)] $\xx \subseteq \dzn{S} \cap
\bigcap_{\omega \in \varOmega}\dzn{S_{\omega}}$,
   \item[(ii)] $\ff := \lin \bigcup_{n=0}^\infty
S^n(\xx)$ is a core of $S$,
   \item[(iii)]  $\is{S^m x}{S^n y} =
\lim_{\omega \in \varOmega} \is{S_{\omega}^m
x}{S_{\omega}^n y}$ for all $x,y \in \xx$ and $m,n \in
\zbb_+$.
   \end{enumerate}
Then $S$ is subnormal.
   \end{thm}
   \begin{proof}
Set $\ff_{\omega}=\lin \bigcup_{n=0}^\infty
S_{\omega}^n(\xx)$ for $\omega \in \varOmega$. It is
clear that $S_{\omega}|_{\ff_{\omega}}$ is a subnormal
operator in $\overline{\ff_{\omega}}$ with an
invariant domain.

Take a finite system $\{a_{p,q}^{i,j}\}_{p,q = 0,
\ldots, n}^{i,j=1, \ldots, m}$ of complex numbers such
that
   \begin{align*}
\sum_{i,j=1}^m \sum_{p,q=0}^n a_{p,q}^{i,j} \lambda^p
\bar \lambda^q z_i \bar z_j \Ge 0, \quad \lambda, z_1,
\ldots, z_m \in \cbb.
   \end{align*}
Let $f_1, \ldots, f_m$ be arbitrary vectors in $\ff$.
Then for every $i \in \{1, \ldots, m\}$, there exists
a positive integer $r$ and a system
$\{\zeta_{x,k}^{(i)}\colon x \in \xx, k= 1, \ldots,
r\}$ of complex numbers such that the set $\{x \in
\xx\colon \zeta_{x,k}^{(i)} \neq 0\}$ is finite for
every $k\in \{1, \ldots, r\}$, and $f_i = \sum_{x \in
\xx}\sum_{k=1}^r \zeta_{x,k}^{(i)} S^k x$. Set
$f_{i,\omega} = \sum_{x \in \xx}\sum_{k=1}^r
\zeta_{x,k}^{(i)} S_{\omega}^k x$ for $i \in \{1,
\ldots, m\}$ and $\omega \in \varOmega$. Then
$f_{i,\omega} \in \ff_{\omega}$ for all $i \in \{1,
\ldots, m\}$ and $\omega \in \varOmega$. Applying
\cite[Theorem 21]{c-s-sz} to the subnormal operators
$S_{\omega}|_{\ff_{\omega}}$, we get
\allowdisplaybreaks
   \begin{multline*}
\sum_{i,j=1}^m \sum_{p,q=0}^n a_{p,q}^{i,j} \is{S^p
f_i} {S^q f_j} = \sum_{i,j=1}^m \sum_{p,q=0}^n
\sum_{x,y \in \xx} \sum_{k,l=1}^r a_{p,q}^{i,j}
\zeta_{x,k}^{(i)} \overline{\zeta_{y,l}^{(j)}}
\is{S^{p+k} x} {S^{q+l} y}
   \\
\overset{{\rm (iii)}}= \lim_{\omega \in \varOmega}
\sum_{i,j=1}^m \sum_{p,q=0}^n \sum_{x,y \in \xx}
\sum_{k,l=1}^r a_{p,q}^{i,j} \zeta_{x,k}^{(i)}
\overline{\zeta_{y,l}^{(j)}} \is{S_{\omega}^{p+k}
x} {S_{\omega}^{q+l} y}
   \\
= \lim_{\omega \in \varOmega} \sum_{i,j=1}^m
\sum_{p,q=0}^n a_{p,q}^{i,j} \is{S_{\omega}^p
f_{i,\omega}} {S_{\omega}^q f_{j,\omega}} \Ge 0.
   \end{multline*}
This means that the operator $S|_{\ff}$ satisfies
condition (ii) of \cite[Theorem 21]{c-s-sz}. Since
$S|_{\ff}$ has an invariant domain, we deduce from
\cite[Theorem 21]{c-s-sz} that $S|_{\ff}$ is
subnormal. Combining the latter with the assumption
that $\ff$ is a core of $S$, we see that $S$ itself is
subnormal. This completes the proof.
   \end{proof}
We say that a densely defined operator $S$ in a
complex Hilbert space $\hh$ is {\em cyclic} with
a {\em cyclic vector} $e \in \hh$ if $e \in
\dzn{S}$ and $\lin\{S^n e\colon n=0,1, \ldots\}$
is a core of $S$.
   \begin{cor}
Let $\{S_{\omega}\}_{\omega \in \varOmega}$ be a net
of subnormal operators in a complex Hilbert space
$\hh$ and let $S$ be a cyclic operator in $\hh$ with a
cyclic vector $e$ such that
   \begin{enumerate}
   \item[(i)] $e \in
\bigcap_{\omega \in \varOmega}\dzn{S_{\omega}}$,
   \item[(ii)]  $\is{S^m e}{S^n e} =
\lim_{\omega \in \varOmega} \is{S_{\omega}^m
e}{S_{\omega}^n e}$ for all $m,n \in \zbb_+$.
   \end{enumerate}
Then $S$ is subnormal.
   \end{cor}
   \subsection{\label{subs1}Necessity}
Let us recall a well-known fact that
$C^\infty$-vectors of a subnormal operator always
generate Stieltjes moment sequences.
   \begin{pro}\label{necess-gen}
If $S$ is a subnormal operator in a complex
Hilbert space $\hh$, then $\dzn{S} = \sti{S}$,
where $\sti{S}$ stands for the set of all vectors
$f \in \dzn{S}$ such that the sequence $\{\|S^n
f\|^2\}_{n=0}^\infty$ is a Stieltjes moment
sequence.
   \end{pro}
   \begin{proof}
Let $N$ be a normal extension of $S$ acting in a
complex Hilbert space $\kk \supseteq \hh$ and let $E$
be the spectral measure of $N$. Define the mapping
$\phi \colon \cbb \to \rbb_+$ by $\phi(z)=|z|^2$, $z
\in \cbb$. Since evidently $\dzn{S} \subseteq
\dzn{N}$, we deduce from the measure transport theorem
(cf.\ \cite[Theorem 5.4.10]{b-s}) that for every $f
\in \dzn{S}$,
   \begin{align*}
\|S^n f\|^2 = \|N^n f\|^2 &= \Big\|\int_{\cbb} z^n
E(\D z)f\Big\|^2
   \\
   &= \int_{\cbb} \phi(z)^n \is{E(\D z)f}{f} =
\int_0^\infty t^n \is{F(\D t)f}{f}, \quad n \in
\zbb_+,
   \end{align*}
where $F$ is the spectral measure on $\rbb_+$ given by
$F(\sigma) = E(\phi^{-1}(\sigma))$ for $\sigma \in
\borel{\rbb_+}$. This implies that $\dzn{S} \subseteq
\sti{S}$.
   \end{proof}
It follows from Proposition \ref{necess-gen} that if
$S$ is a subnormal operator with invariant domain,
then $S$ is densely defined and $\dz{S}=\sti{S}$. One
might expect that the reverse implication holds as
well. This is really the case for bounded operators
(cf.\ \cite{Lam}) and for some unbounded operators
that have sufficiently many analytic vectors (cf.\
\cite[Theorem 7]{StSz1}). In Section \ref{cfs} we show
that this is also the case for weighted shifts on
directed trees that have sufficiently many
quasi-analytic vectors (see Theorem \ref{main-0}).
However, in general, this is not the case. Indeed, one
can construct a densely defined operator $N$ in a
complex Hilbert space $\hh$ which is not subnormal and
which has the following properties (see
\cite{Cod,Sch,sto-ark}):
   \begin{gather}  \label{fn1}
N(\dz{N}) \subseteq \dz{N}, \, \dz{N} \subseteq
\dz{N^*}, \, N^*(\dz{N}) \subseteq \dz{N}
   \\
\text{ and } N^*Nf = NN^*f \text{ for all } f\in
\dz{N}. \label{fn2}
   \end{gather}
We show that for such $N$, $\dz{N}=\sti{N}$.
Indeed, by \eqref{fn1} and \eqref{fn2}, we have
   \begin{align*}
\sum_{k,l=0}^n \|N^{k+l}f\|^2 \alpha_k
\overline{\alpha_l} = \sum_{k,l=0}^n
\is{(N^*N)^{k+l}f}{f} \alpha_k \overline{\alpha_l} =
\Big\|\sum_{k=0}^n \alpha_k (N^*N)^k f\Big\|^2 \Ge 0,
   \end{align*}
for all $f \in \dz{N}$, $n \in \zbb_+$ and
$\alpha_0,\ldots, \alpha_n \in \cbb$, which means that the
sequence $\{\|N^{n}f\|^2\}_{n=0}^\infty$ is positive
definite for every $f \in \dz{N}$. Replacing $f$ by $Nf$,
we see that the sequence $\{\|N^{n+1}f\|^2\}_{n=0}^\infty$
is positive definite for every $f \in \dz{N}$. Applying the
Stieltjes theorem, we conclude that $\dz{N}=\sti{N}$.
   \section{Towards Subnormality of Weighted Shifts}
   \subsection{A consistency condition}
Applying Proposition \ref{necess-gen}, we get.
   \begin{pro}\label{necess}
Let $\slam$ be a weighted shift on a directed tree
$\tcal$ with weights $\lambdab=\{\lambda_v\}_{v \in
V^\circ}$ such that $\escr \subseteq \dzn{\slam}$. If
$\slam$ is subnormal, then for every $u \in V$ the
sequence $\{\|\slam^n e_u\|^2\}_{n=0}^\infty$ is a
Stieltjes moment sequence.
   \end{pro}
The converse of the implication in Proposition
\ref{necess} is valid for bounded weight\-ed shifts on
directed trees (the unbounded case is discussed in
Theorem \ref{main-0}).
   \begin{thm}[\mbox{\cite[Theorem 6.1.3]{j-j-s}}]
\label{charsub} Let $\slam \in \ogr{\ell^2(V)}$ be a
weighted shift on a directed tree $\tcal$ with weights
$\lambdab = \{\lambda_v\}_{v \in V^\circ}$. Then
$\slam$ is subnormal if and only if $\{\|\slam^n
e_u\|^2\}_{n=0}^\infty$ is a Stieltjes moment sequence
for every $u \in V$.
   \end{thm}
If $\slam$ is a subnormal weighted shift on a directed
tree $\tcal$, then in view of Proposition \ref{necess}
we can attach to each vertex $u \in V$ a representing
measure $\mu_u$ of the Stieltjes moment sequence
$\{\|\slam^n e_u\|^2\}_{n=0}^\infty$ (of course, since
the sequence $\{\|\slam^n e_u\|^2\}_{n=0}^\infty$ is
not determinate in general, we have to choose one of
them); note that any such $\mu_u$ is a probability
measure. Hence, it is tempting to find relationships
between these representing measures. This has been
done in the case of bounded weighted shifts in
\cite[Lemma 6.1.10]{j-j-s}. What is stated below is an
adaptation of this lemma (and its proof) to the
unbounded case. As opposed to the bounded case,
implication $1^\circ \Rightarrow 2^\circ$ of Lemma
\ref{charsub2} below is not true in general (cf.\
\cite{j-j-s4}).
   \begin{lem} \label{charsub2}
Let $\slam$ be a weighted shift on a directed tree
$\tcal$ with weights $\lambdab=\{\lambda_v\}_{v \in
V^\circ}$ such that $\escr \subseteq \dzn{\slam}$. Let
$u \in V^\prime$. Suppose that for every $v \in \dzi
u$ the sequence $\{\|\slam^n e_v\|^2\}_{n=0}^\infty$
is a Stieltjes moment sequence with a representing
measure $\mu_v$. Consider the following two
conditions\,\footnote{\;\;We adhere to the standard
convention that $0 \cdot \infty = 0$; see also
footnote \ref{foot}.}{\em :}
   \begin{enumerate}
   \item[$1^\circ$] $\{\|\slam^n e_u\|^2\}_{n=0}^\infty$
is a Stieltjes moment sequence,
   \item[$2^\circ$]   $\slam$ satisfies
the consistency condition at the vertex $u$, i.e.,
   \begin{align} \label{alanconsi}
\sum_{v \in \dzi{u}} |\lambda_v|^2 \int_0^\infty \frac
1 s\, \D \mu_v(s) \Le 1.
   \end{align}
   \end{enumerate}
   Then the following assertions are valid{\em :}
   \begin{enumerate}
   \item[(i)] if $2^\circ$ holds, then so does
$1^\circ$ and the positive Borel measure $\mu_u$ on
$\rbb_+$ defined by
   \begin{align} \label{muu+}
\mu_u(\sigma) = \sum_{v \in \dzi u} |\lambda_v|^2
\int_\sigma \frac 1 s \D \mu_v(s) + \varepsilon_u
\delta_0(\sigma), \quad \sigma \in \borel{\rbb_+},
      \end{align}
with
   \begin{align} \label{muu++}
\varepsilon_u=1 - \sum_{v \in \dzi u} |\lambda_v|^2
\int_0^\infty \frac 1 s \D \mu_v(s)
   \end{align}
is a representing measure of $\{\|\slam^n
e_u\|^2\}_{n=0}^\infty$,
   \item[(ii)] if $1^\circ$ holds and $\{\|\slam^{n+1}
e_u\|^2\}_{n=0}^\infty$ is determinate, then $2^\circ$
holds, the Stieltjes moment sequence $\{\|\slam^n
e_u\|^2\}_{n=0}^\infty$ is determinate and its unique
representing measure $\mu_u$ is given by \eqref{muu+}
and \eqref{muu++}.
   \end{enumerate}
   \end{lem}
   \begin{proof}
Define the positive Borel measure $\mu$ on $\rbb_+$ by
   \begin{align*}
\mu(\sigma) = \sum_{v \in \dzi u} |\lambda_v|^2
\mu_v(\sigma), \quad \sigma \in \borel{\rbb_+}.
   \end{align*}
It is a matter of routine to show that
   \begin{align}  \label{leb2}
\int_0^\infty f \D \mu = \sum_{v \in \dzi u}
|\lambda_v|^2 \int_0^\infty f \D \mu_v
   \end{align}
for every Borel function $f\colon {[0,\infty)} \to
[0,\infty]$. Using the inclusion $\escr \subseteq
\dzn{\slam}$ and applying Lemma \ref{lem4}\,(iii)
twice, we obtain \allowdisplaybreaks
   \begin{align*}
\|\slam^{n+1} e_u\|^2 & \hspace{2.2ex}= \sum_{w \in
\dzin{n+1}{u}} |\lambda_{u\mid w}|^2
   \\
& \hspace{.4ex} \overset{\eqref{dzinn2}}= \sum_{v \in
\dzi u} \sum_{w \in \dzin{n} v} |\lambda_{u\mid w}|^2
\notag
   \\
& \hspace{.4ex} \overset{\eqref{recfor2}}= \sum_{v \in
\dzi u} |\lambda_v|^2 \sum_{w \in \dzin{n} v}
|\lambda_{v\mid w}|^2 \notag
   \\
& \hspace{2.2ex} =\sum_{v \in \dzi u} |\lambda_v|^2
\|\slam^n e_v\|^2, \quad n \in \zbb_+. \notag
   \end{align*}
This implies that
   \begin{align*}
\|\slam^{n+1} e_u\|^2 = \sum_{v \in \dzi u}
|\lambda_v|^2 \int_0^\infty s^n \, \D \mu_v(s)
\overset{\eqref{leb2}}= \int_0^\infty s^n \D \mu(s),
\quad n \in \zbb_+.
   \end{align*}
   Hence the sequence $\{\|\slam^{n+1}
   e_u\|^2\}_{n=0}^\infty$ is a Stieltjes moment
   sequence with a representing measure $\mu$.

Set $t_n = \|\slam^{n+1} e_u\|^2$ for $n \in \zbb_+$,
and $t_{-1}=1$. Note that
   \begin{align*}
t_{n-1}=\|\slam^{n} e_u\|^2, \quad n \in \zbb_+.
   \end{align*}

Suppose that $2^\circ$ holds. Then, by
\eqref{alanconsi} and \eqref{leb2}, we have
$\int_0^\infty \frac 1 s \D \mu(s) \Le 1$. Applying
implication (iii)$\Rightarrow$(i) of Lemma \ref{bext},
we see that $1^\circ$ holds, and, by \eqref{leb2}, the
measure $\mu_u$ defined by \eqref{muu+} and
\eqref{muu++} is a representing measure of the
Stieltjes moment sequence $\{\|\slam^n
e_u\|^2\}_{n=0}^\infty$.

Suppose now that $1^\circ$ holds and the Stieltjes
moment sequence $\{\|\slam^{n+1}
e_u\|^2\}_{n=0}^\infty$ is determinate. It follows
from implication (i)$\Rightarrow$(iii) of Lemma
\ref{bext} that there is a representing measure
$\mu^\prime$ of $\{\|\slam^{n+1}
e_u\|^2\}_{n=0}^\infty$ such that $\int_0^\infty \frac
1 s \D \mu^\prime(s) \Le 1$. Since $\{\|\slam^{n+1}
e_u\|^2\}_{n=0}^\infty$ is determinate, we get
$\mu^\prime=\mu$, which implies $2^\circ$. The
remaining part of assertion (ii) follows from the last
assertion of Lemma \ref{bext}.
   \end{proof}
Now we prove that the determinacy of appropriate
Stieltjes moment sequences attached to a weighted
shift on a directed tree implies the existence of a
consistent system of measures (see also Theorem
\ref{necessdet2}). As shown in \cite{j-j-s4}, Lemma
\ref{2necess+} below is no longer true if the
assumption on determinacy is dropped (by Lemma
\ref{lem3}\,(iv), the converse of Lemma \ref{2necess+}
is true without assuming determinacy).
   \begin{lem} \label{2necess+}
 Let $\slam$ be a weighted shift on a directed tree
 $\tcal$ with weights $\lambdab=\{\lambda_v\}_{v \in
 V^\circ}$ such that $\escr \subseteq \dzn{\slam}$.
 Assume that for every $u \in V^\prime$, the sequence
 $\{\|\slam^{n} e_u\|^2\}_{n=0}^\infty$ is a Stieltjes
 moment sequence, and that the Stieltjes moment
 sequence $\{\|\slam^{n+1} e_u\|^2\}_{n=0}^\infty$
 $($cf.\ \eqref{st+1}$)$ is determinate. Then there
 exist a system $\{\mu_u\}_{u \in V}$ of Borel
 probability measures on $\rbb_+$ and a system
 $\{\varepsilon_u\}_{u \in V}$ of nonnegative real
 numbers that satisfy \eqref{muu+} for every $u \in
 V$.
   \end{lem}
   \begin{proof}
By Lemma \ref{bext}, the Stieltjes moment sequence
$\{\|\slam^n e_u\|^2\}_{n=0}^\infty$ is determinate
for every $u\in V^\prime$. For $u \in V^\prime$, we
denote by $\mu_u$ the unique representing measure of
$\{\|\slam^n e_u\|^2\}_{n=0}^\infty$. If $u \in V
\setminus V^\prime$, then we put $\mu_u=\delta_0$.
Using Lemma \ref{charsub2}\,(ii), we verify that the
system $\{\mu_u\}_{u \in V}$ satisfies \eqref{muu+}
with $\{\varepsilon_u\}_{u \in V}$ defined by
\eqref{muu++}. This completes the proof.
   \end{proof}
   \subsection{Consistent systems of measures}
In this section we prove some important
properties of consistent systems of Borel
probability measures on $\rbb_+$ attached to a
directed tree. They will be used in the proof of
Theorem \ref{main}.
   \begin{lem} \label{lem1}
Let $\tcal$ be a directed tree. Suppose that
$\{\lambda_v\}_{v \in V^\circ}$ is a system of
complex numbers, $\{\varepsilon_v\}_{v \in V}$ is
a system of nonnegative real numbers and
$\{\mu_v\}_{v \in V}$ is a system of Borel
probability measures on $\rbb_+$ satisfying
\eqref{muu+} for every $u \in V$. Then the
following assertions hold\/{\em :}
   \begin{enumerate}
   \item[(i)] for every $u \in V$,
$\sum_{v\in \dzi{u}} |\lambda_v|^2 \int_0^\infty \frac
1 s \D \mu_v(s) \Le 1$ and
   \begin{align*}
\varepsilon_u = 1 - \sum_{v\in \dzi{u}} |\lambda_v|^2
\int_0^\infty \frac 1 s \D \mu_v(s),
   \end{align*}
   \item[(ii)] for every $u \in V$,  $\mu_u(\{0\})= 0$
if and only if $\varepsilon_u=0$,
   \item[(iii)] for every $v \in V^\circ$, if $\lambda_v \neq
0$, then $\mu_v(\{0\})=0$,
   \item[(iv)]  for every $u \in V$,
   \begin{align}  \label{wz2}
\mu_u(\sigma) = \sum_{v\in \dzin{n}{u}}
|\lambda_{u\mid v}|^2 \int_{\sigma} \frac 1 {s^n} \D
\mu_v(s) + \varepsilon_u \delta_0(\sigma), \quad
\sigma \in \borel{\rbb_+}, \, n \Ge 1.
   \end{align}
   \end{enumerate}
   \end{lem}
   \begin{proof}
   (i) Substitute $\sigma = \rbb_+$ into \eqref{muu+}
and note that $\mu_u(\rbb_+)=1$.

   (ii) \& (iii) Substitute $\sigma = \{0\}$ into
\eqref{muu+}.

   (iv) We use induction on $n$. The case of $n=1$
coincides with \eqref{muu+}. Suppose that \eqref{wz2}
is valid for a fixed integer $n\Ge 1$. Then combining
\eqref{muu+} with \eqref{wz2}, we see that
   \begin{multline} \label{wz3}
\mu_u(\sigma) = \sum_{v\in \dzin{n}{u}}
|\lambda_{u\mid v}|^2 \sum_{w \in \dzi{v}}
|\lambda_w|^2\int_{\sigma} \frac 1 {s^{n+1}} \D
\mu_w(s)
   \\
   + \sum_{v\in \dzin{n}{u}} |\lambda_{u\mid v}|^2
\int_{\sigma} \frac 1 {s^n} \D (\varepsilon_v
\delta_0)(s) + \varepsilon_u \delta_0(\sigma), \quad
\sigma \in \borel{\rbb_+}.
   \end{multline}
Since $\mu_u$ is a finite positive measure and
$n\Ge 1$, we deduce from \eqref{wz3} that
$\varepsilon_v=0$ whenever $\lambda_{u\mid v}
\neq 0$, and thus
   \begin{align}  \label{wz4}
\sum_{v\in \dzin{n}{u}} |\lambda_{u\mid v}|^2
\int_{\sigma} \frac 1 {s^n} \D (\varepsilon_v
\delta_0)(s)=0.
   \end{align}
It follows from \eqref{wz3} and \eqref{wz4} that
   \begin{align*}
\mu_u(\sigma) = \sum_{v\in \dzin{n}{u}} \sum_{w \in
\dzi{v}} |\lambda_{u\mid v}\lambda_w|^2\int_{\sigma}
\frac 1 {s^{n+1}} \D \mu_w(s) + \varepsilon_u
\delta_0(\sigma)
      \\
      \overset{\eqref{num1}\&\eqref{num2}}= \sum_{w\in
\dzin{n+1}{u}} |\lambda_{u\mid w}|^2 \int_{\sigma}
\frac 1 {s^{n+1}} \D \mu_w(s) + \varepsilon_u
\delta_0(\sigma).
   \end{align*}
This completes the proof.
   \end{proof}
   \begin{lem} \label{lem3}
Let $\tcal$ be a directed tree. Suppose that
$\lambdab=\{\lambda_v\}_{v \in V^\circ}$ is a
system of complex numbers, $\{\varepsilon_v\}_{v
\in V}$ is a system of nonnegative real numbers
and $\{\mu_v\}_{v \in V}$ is a system of Borel
probability measures on $\rbb_+$ satisfying
\eqref{muu+} for every $u \in V$. Let $\slam$ be
a weighted shift on the directed tree $\tcal$
with weights $\lambdab$. Then the following
assertions hold\/{\em :}
   \begin{enumerate}
   \item[(i)] for all $u \in V$ and $n \in \nbb$,
   \begin{align}  \label{wz5}
\int_0^\infty s^n \D \mu_u(s) = \sum_{v\in
\dzin{n}{u}} |\lambda_{u\mid v}|^2,
   \end{align}
   \item[(ii)] if $\dzin{n}{u}=\varnothing$ for some
$u\in V$ and $n\in \nbb$, then $\mu_v=\delta_0$ for
all $v \in \des{u}$,
   \item[(iii)] $\escr \subseteq \dzn{\slam}$ if and only
if $\int_0^\infty s^n \D \mu_u(s) < \infty$ for all $n
\in \zbb_+$ and $u \in V$,
   \item[(iv)] if $\escr \subseteq \dzn{\slam}$, then
for all $u \in V$ and $n \in \zbb_+$,
   \begin{align}    \label{wz6}
\|\slam^n e_u\|^2 = \int_0^\infty s^n \D \mu_u(s),
   \end{align}
   \item[(v)] $\slam \in \ogr{\ell^2(V)}$ if and
only if there exists a real number $M \Ge 0$ such that
$\supp \mu_u \subseteq [0,M]$ for every $u \in V$.
   \end{enumerate}
   \end{lem}
   \begin{proof}
   (i) Substituting $\sigma=\{0\}$ into
\eqref{wz2}, we see that for every $v\in
\dzin{n}{u}$, either $\lambda_{u\mid v} = 0$, or
$\lambda_{u\mid v} \neq 0$ and $\mu_v(\{0\})=0$.
This and \eqref{wz2} lead to \eqref{wz5}.

   (ii) It follows from \eqref{wz5} that
$\int_0^\infty s^n \D \mu_u(s) = 0$ (recall the
convention that $\sum_{v\in\varnothing} x_v=0$). This
and $n \Ge 1$ implies that $\mu_u((0,\infty))=0$.
Since $\mu_u(\rbb_+) = 1$, we deduce that
$\mu_u=\delta_0$.

If $v \in \des{u}\setminus \{u\}$, then by
\eqref{num3} there exists $k\in \nbb$ such that
$v\in \dzin{k}{u}$. Since $\dzi{\cdot}$ is a
monotonically increasing set-function, we infer
from \eqref{chmn} that $\dzin{n}{v} \subseteq
\dzin{n+k}{u}=\varnothing$. By the previous
argument applied to $v$ in place of $u$, we get
$\mu_v=\delta_0$.

   Assertions (iii) and (iv) follow from (i) and Lemma
\ref{lem4}.

   (v) To prove the ``only if'' part, note that
   \begin{align*}
\lim_{n\to\infty} \Big(\int_0^\infty s^{n} \D
\mu_u(s)\Big)^{1/n} \overset{\eqref{wz6}}=
\lim_{n\to\infty} (\|\slam^n e_u\|^{1/n})^2 \Le
\|\slam\|^2,
   \end{align*}
which implies that $\supp \mu_u \subseteq
[0,\|\slam\|^2]$ (cf.\ \cite[page 71]{Rud}). The proof
of the converse implication goes as follows. In view
of \eqref{wz5}, we have
   \begin{align*}
\sum_{v\in \dzi{u}} |\lambda_{v}|^2 = \int_0^\infty s
\D \mu_u(s) \Le M, \quad u \in V,
   \end{align*}
which, by \cite[Proposition 3.1.8]{j-j-s}, implies
that $\slam \in \ogr{\ell^2(V)}$ and $\|\slam\| \Le
\sqrt{M}$.
   \end{proof}
   \section{Criteria for Subnormality of
Weighted Shifts}
   \subsection{Arbitrary weights}
   After all these preparations we can prove the main
criterion for subnormality of unbounded weighted
shifts on directed trees. It is written in terms of
consistent systems of measures.
   \begin{thm} \label{main}
 Let $\slam$ be a weighted shift on a directed
tree $\tcal$ with weights
$\lambdab=\{\lambda_v\}_{v \in V^\circ}$ such
that $\escr \subseteq \dzn{\slam}$. Suppose that
there exist a system $\{\mu_v\}_{v \in V}$ of
Borel probability measures on $\rbb_+$ and a
system $\{\varepsilon_v\}_{v \in V}$ of
nonnegative real numbers that satisfy
\eqref{muu+} for every $u \in V$. Then $\slam$ is
subnormal.
   \end{thm}
   \begin{proof}
   For a fixed positive integer $i$, we define the
system $\lambdabi =\big\{\lambdai{v}\big\}_{v \in
V^\circ}$ of complex numbers, the system
$\big\{\mui{v}\big\}_{v \in V}$ of Borel probability
measures on $\rbb_+$ and the system
$\big\{\varepsiloni{v}\big\}_{v \in V}$ of nonnegative
real numbers by
   \allowdisplaybreaks
   \begin{align}  \label{wzd1}
\lambdai{v} & =
   \begin{cases}
   \lambda_v
\sqrt{\cfrac{\mu_v([0,i])}{\mu_{\pa{v}}([0,i])}} &
\text{ if } \mu_{\pa{v}}([0,i]) > 0,
   \\[1.5ex]
0 & \text{ if } \mu_{\pa{v}}([0,i]) = 0,
   \end{cases}
\quad v \in V^\circ,
   \\     \label{wzd2}
\mui{v}(\sigma) & =
   \begin{cases}
\cfrac{\mu_v(\sigma \cap [0,i])}{\mu_v([0,i])} &
\text{ if } \mu_v([0,i]) > 0,
   \\[1.5ex]
\delta_0(\sigma) & \text{ if } \mu_v([0,i]) = 0,
   \end{cases}
\quad \sigma \in \borel{\rbb_+},\, v \in V,
   \\  \label{wzd3}
\varepsiloni{v} & =
   \begin{cases}
\cfrac{\varepsilon_v}{\mu_v([0,i])} & \text{ if }
\mu_v([0,i]) > 0,
   \\[1.5ex]
1 & \text{ if } \mu_v([0,i])=0,
   \end{cases}
\quad v \in V.
   \end{align}
Our first goal is to show that the following equality
holds for all $u \in V$ and $i\in \nbb$,
   \begin{align} \label{wz1J}
\mui{u}(\sigma) = \sum_{v\in \dzi{u}} |\lambdai{v}|^2
\int_{\sigma} \frac 1 s \D \mui{v}(s) +
\varepsiloni{u} \delta_0(\sigma), \quad \sigma \in
\borel{\rbb_+}.
   \end{align}
For this fix $u \in V$ and $i\in \nbb$. If
$\mu_u([0,i])=0$, then, according to our definitions,
we have $\lambdai{v}=0$ for all $v \in \dzi{u}$,
$\mui{u}=\delta_0$ and $\varepsiloni{u}=1$, which
means that the equality \eqref{wz1J} holds. Consider
now the case of $\mu_u([0,i])>0$. It follows from
\eqref{muu+} that
   \begin{align}   \label{muu}
\mu_u(\sigma \cap [0,i]) = \sum_{v\in \dzi{u}}
|\lambda_v|^2 \int_{\sigma \cap [0,i]} \frac 1 s \D
\mu_v(s) + \varepsilon_u \delta_0(\sigma), \quad
\sigma \in \borel{\rbb_+}.
   \end{align}
If $v \in \dzi{u}$ (equivalently:\ $u=\pa{v}$), then
by \eqref{wzd1} and \eqref{wzd2} we have
   \begin{align} \label{muu2}
   \begin{aligned}
\frac{|\lambda_v|^2}{\mu_u([0,i])} \int_{\sigma \cap
[0,i]} \frac 1 s \D \mu_v(s) & =
   \begin{cases}
   |\lambdai{v}|^2 \int_{\sigma} \frac 1 s \D
\mui{v}(s) & \text{ if } \mu_v([0,i]) > 0,
   \\[1ex]
   0 & \text{ if } \mu_v([0,i])=0,
   \end{cases}
   \\
   & = |\lambdai{v}|^2 \int_{\sigma} \frac 1 s \D
\mui{v}(s),
    \end{aligned}
   \end{align}
where the last equality holds because
$\lambdai{v}=0$ whenever $\mu_v([0,i])=0$.
Dividing both sides of \eqref{muu} by
$\mu_u([0,i])$ and using \eqref{muu2}, we obtain
\eqref{wz1J}.

Let $S_{\lambdabi}$ be the weighted shift on $\tcal$
with weights $\lambdabi$. Since, by \eqref{wzd2},
$\supp \mui{u} \subseteq [0,i]$ for every $u \in V$,
we infer from \eqref{wz1J} and Lemma~ \ref{lem3}\,(v),
applied to the triplet $(\lambdabi, \{\mui{v}\}_{v \in
V}, \{\varepsiloni{v}\}_{v \in V})$, that
$S_{\lambdabi}\in \ogr{\ell^2(V)}$. In turn,
\eqref{wz1J} and Lemma \ref{lem3}\,(iv) (applied to
the same triplet) imply that for every $u \in V$,
$\{\|S_{\lambdabi}^n e_u\|^2\}_{n=0}^\infty$ is a
Stieltjes moment sequence (with a representing measure
$\mui{u}$). Hence, by Theorem \ref{charsub}, the
operator $S_{\lambdabi}$ is subnormal.

Since $\mu_u$, $u \in V$, are Borel probability
measures on $\rbb_+$, we have
   \begin{align}  \label{lim1}
\lim_{i \to \infty} \mu_u([0,i]) = 1, \quad u \in V.
   \end{align}
Hence, for every $u \in V$ there exists a positive
integer $\kappa_{u}$ such that
   \begin{align}   \label{as1}
\mu_{u}([0,i]) > 0,\quad i \in \nbb, \, i \Ge
\kappa_{u}.
   \end{align}
Note that
   \begin{align} \label{wzj}
\lim_{i \to \infty} \lambdai{v} = \lambda_v, \quad v
\in V^\circ.
   \end{align}
Indeed, if $v \in V^\circ$, then \eqref{wzd1} and
\eqref{as1} yield $\lambdai{v}=\lambda_v
\sqrt{\frac{\mu_v([0,i])} {\mu_{\pa{v}}([0,i])}}$
for all integers $i \Ge \kappa_{\pa{v}}$. This,
combined with \eqref{lim1}, gives \eqref{wzj}. By
\eqref{wzd2}, \eqref{as1}, \eqref{wz1J} and Lemma
\ref{lem3}\,(iv), applied to $S_{\lambdabi}$, we
have
   \begin{align*}
\|S_{\lambdabi}^n e_u\|^2 = \int_0^\infty s^n \D
\mui{u}(s) = \frac{1}{\mu_u([0,i])} \int_{[0,i]} s^n
\D \mu_u(s), \quad n \in \zbb_+,\, i \Ge \kappa_{u},
\, u \in V.
   \end{align*}
This, together with \eqref{lim1} and Lemma
\ref{lem3}\,(iv), now applied to $\slam$, implies
that
   \begin{align}    \label{limsti}
\lim_{i \to \infty} \|S_{\lambdabi}^n e_u\|^2 =
\int_0^\infty s^n \D \mu_u(s) = \|\slam^n e_u\|^2,
\quad n \in \zbb_+,\, u \in V.
   \end{align}
It follows from \eqref{wzj}, \eqref{limsti} and
Proposition \ref{potegi} that \eqref{slim+} holds.
According to \cite[Proposition 3.1.3\,(vi)]{j-j-s},
$\escr$ is a core of $\slam$. Hence $\lin
\bigcup_{n=0}^\infty \slam^n(\escr)$ is a core of
$\slam$ as well. Applying \eqref{slim+} and Theorem
\ref{tw1} to the operators
$\{S_{\lambdabi}\}_{i=1}^\infty$ and $\slam$ with
$\xx:=\{e_u\colon u \in V\}$ completes the proof of
Theorem \ref{main}.
   \end{proof}
   \begin{rem}
In the proof of Theorem \ref{main} we have used
Proposition \ref{potegi} which provides a general
criterion for the validity of the approximation
procedure \eqref{slim+}. However, if the
approximating triplets $(\lambdabi,
\{\mui{v}\}_{v \in V}, \{\varepsiloni{v}\}_{v \in
V})$, $i=1,2,3, \ldots$, are defined as in
\eqref{wzd1}, \eqref{wzd2} and \eqref{wzd3}, then
   \begin{align} \label{slim2}
\lim_{i\to \infty} S_{\lambdabi}^n e_u =
S_{\lambdab}^n e_u, \quad u \in V, \, n \in
\zbb_+.
   \end{align}
To prove this, we first show that for all $u \in V$
and $i \Ge \kappa_u$ (see \eqref{as1}),
   \begin{align} \label{liuup}
\lambdai{u \mid u^\prime} = \lambda_{u \mid
u^\prime} \; \sqrt{\frac{\mu_{u^\prime}
([0,i])}{\mu_u([0,i])}}, \quad u^\prime \in
\dzin{n}{u}, \, n \in \zbb_+.
   \end{align}
Indeed, if $n=0$, then \eqref{liuup} holds.
Suppose that $n\Ge 1$. If
$\mu_{\paa(u^\prime)}([0,i])=0$, then $n\Ge 2$
and, by \eqref{wzd1}, $\lambdai{u^\prime} = 0$,
which implies that $\lambdai{u \mid u^\prime}=0$.
Since $\mu_{\paa(u^\prime)}([0,i])=0$, we deduce
from \eqref{muu+} (applied to $u=\paa(u^\prime)$)
that either $\lambda_{u^\prime}=0$, or
$\mu_{u^\prime} ([0,i]) = 0$. In both cases, the
right-hand side of \eqref{liuup} vanishes, and so
\eqref{liuup} holds. In turn, if
$\mu_{\paa(u^\prime)}([0,i]) > 0$, then we can
define
   \begin{align*}
j_0 = \min \Big\{j \in \{1, \ldots, n\}\colon
\mu_{\paa^k(u^\prime)}([0,i]) > 0 \text{ for all
} k=1, \ldots, j\Big\}.
   \end{align*}
Clearly, $1 \Le j_0 \Le n$. First, we consider
the case where $j_0 < n$. Since, by \eqref{as1},
$\mu_u([0,i])> 0$, we must have $j_0 \Le n-2$.
Thus $\mu_{\paa^{j_0+1}(u^\prime)}([0,i]) = 0$,
which together with \eqref{luv} and \eqref{wzd1}
implies that the left-hand side of \eqref{liuup}
vanishes. Since
$\mu_{\paa^{j_0+1}(u^\prime)}([0,i]) = 0$ and
$\mu_{\paa^{j_0}(u^\prime)}([0,i]) > 0$, we
deduce from \eqref{muu+} (applied to
$u=\paa^{j_0+1}(u^\prime)$) that
$\lambda_{\paa^{j_0}(u^\prime)}=0$, and so the
right-hand side of \eqref{liuup} vanishes. This
means that \eqref{liuup} is again valid. Finally,
if $j_0=n$, then by \eqref{wzd1} we have
   \begin{align*}
\lambdai{u \mid u^\prime} = \prod_{j=0}^{n-1}
\lambda_{\paa^j(u^\prime)}
\sqrt{\frac{\mu_{\paa^j(u^\prime)([0,i])}}
{\mu_{\paa^{j+1}(u^\prime)([0,i])}}} = \lambda_{u
\mid u^\prime} \; \sqrt{\frac{\mu_{u^\prime}
([0,i])}{\mu_u([0,i])}},
   \end{align*}
which completes the proof of \eqref{liuup}. Now we
show that
   \begin{align} \label{ils}
\lim_{i \to \infty} \is{S_{\lambdab}^n
e_u}{S_{\lambdabi}^n e_u} = \|\slam^n e_u\|^2,
\quad u \in V, \, n \in \zbb_+.
   \end{align}
Indeed, it follows from Lemma \ref{lem4}(ii) and
\eqref{liuup} that
   \begin{multline*}
\is{S_{\lambdab}^n e_u}{S_{\lambdabi}^n e_u} =
\sum_{u^\prime \in \dzin{n}{u}} \lambda_{u\mid
u^\prime} \overline{\lambdai{u \mid u^\prime}}
   \\
= \frac{1}{\sqrt{\mu_u([0,i])}} \sum_{u^\prime
\in \dzin{n}{u}} |\lambda_{u\mid u^\prime}|^2
\sqrt{\mu_{u^\prime}([0,i])}, \quad u \in V, \, n
\in \zbb_+, \, i \Ge \kappa_u.
   \end{multline*}
By applying Lebesgue's monotone convergence theorem
for series, \eqref{lim1} and Lem\-ma \ref{lem4}(iii),
we obtain \eqref{ils}. Since
   \begin{align*}
\|S_{\lambdab}^n e_u - S_{\lambdabi}^n e_u\|^2 =
\|S_{\lambdab}^n e_u\|^2 + \|S_{\lambdabi}^n
e_u\|^2 - 2 \, \mathrm{Re}\is{S_{\lambdab}^n
e_u}{S_{\lambdabi}^n e_u}
   \end{align*}
we infer \eqref{slim2} from \eqref{limsti} and
\eqref{ils}. Clearly \eqref{slim2} implies
\eqref{slim+}.
   \end{rem}
We conclude this section with a general criterion for
subnormality of weighted shifts on directed trees
written in terms of determinacy of Stieltjes moment
sequences.
   \begin{thm} \label{necessdet2}
Let $\slam$ be a weighted shift on a directed tree
$\tcal$ with weights $\lambdab=\{\lambda_v\}_{v \in
V^\circ}$ such that $\escr \subseteq \dzn{\slam}$.
Assume that $\{\|\slam^{n+1} e_u\|^2\}_{n=0}^\infty$
is a determinate Stieltjes moment sequence for every
$u \in V$. Then the following conditions are
equivalent{\em :}
   \begin{enumerate}
   \item[(i)] $\slam$ is subnormal,
   \item[(ii)] $\{\|\slam^{n} e_u\|^2\}_{n=0}^\infty$
is a Stieltjes moment sequence for every $u \in V$,
   \item[(iii)] there exist a system $\{\mu_u\}_{u
\in V}$ of Borel probability measures on $\rbb_+$
and a system $\{\varepsilon_u\}_{u \in V}$ of
nonnegative real numbers that satisfy
\eqref{muu+} for every $u \in V$.
   \end{enumerate}
   \end{thm}
   \begin{proof}
(i)$\Rightarrow$(ii) Use Proposition \ref{necess}.

(ii)$\Rightarrow$(iii) Employ Lemma \ref{2necess+}.

(iii)$\Rightarrow$(i) Apply Theorem \ref{main}.
   \end{proof}
Regarding Theorem \ref{necessdet2}, note that by
Proposition \ref{necess}, Lemma \ref{lem3}\,(iv) and
\eqref{st+1} each of the conditions (i), (ii) and
(iii) implies that $\{\|\slam^{n+1}
e_u\|^2\}_{n=0}^\infty$ is a Stieltjes moment sequence
for every $u \in V$.
   \subsection{Nonzero weights}
   As pointed out in \cite[Proposition 5.1.1]{j-j-s}
bounded hyponormal weighted shifts on directed trees
with nonzero weights are always injective. It turns
out that the same conclusion can be derived in the
unbounded case (with almost the same proof). Recall
that a densely defined operator $S$ in $\hh$ is said
to be {\em hyponormal} if $\dz{S} \subseteq \dz{S^*}$
and $\|S^*f\| \Le \|Sf\|$ for all $f \in \dz S$. It is
well-known that subnormal operators are hyponormal
(but not conversely) and that hyponormal operators are
closable and their closures are hyponormal (we refer
the reader to \cite{ot-sch,jj3} for more information
on this subject).
   \begin{pro} 
Let $\tcal$ be a directed tree with $V^\circ \neq
\varnothing$. If $\slam$ is a hyponormal weighted
shift on $\tcal$ whose all weights are nonzero, then
$\tcal$ is leafless. In particular, $\slam$ is
injective and $V$ is infinite and countable.
   \end{pro}
   \begin{proof}
Suppose that, contrary to our claim, $\dzi u =
\varnothing$ for some $u \in V$. We deduce from
\cite[Corollary 2.1.5]{j-j-s} and $V^\circ \neq
\varnothing$ that $u \in V^\circ$. Hence, by
\cite[Propositions 3.1.3 and 3.4.1]{j-j-s}, we have
   \begin{align*}
|\lambda_u|^2 = \|\slam^*e_u\|^2 \Le \|\slam e_u\|^2 =
\sum_{v \in \dzi u} |\lambda_v|^2 = 0,
   \end{align*}
which is a contradiction. Since each leafless directed
tree is infinite, we deduce from \cite[Propositions
3.1.7 and 3.1.10]{j-j-s} that $\slam$ is injective and
$V$ is infinite and countable. This completes the
proof.
   \end{proof}
The sufficient condition for subnormality of weighted
shifts on directed trees stated in Theorem \ref{main}
takes the simplified form for weighted shifts with
nonzero weights. Indeed, if a weighted shift $\slam$
on $\tcal$ with nonzero weights satisfies the
assumptions of Theorem \ref{main}, then, by assertions
(ii) and (iii) of Lemma \ref{lem1}, $\varepsilon_v=0$
for every $v \in V^\circ$. Hence, by applying Theorem
\ref{main}, we get.
   \begin{thm}
   Let $\slam$ be a weighted shift on a directed tree
$\tcal$ with nonzero weights
$\lambdab=\{\lambda_v\}_{v \in V^\circ}$ such that
$\escr \subseteq \dzn{\slam}$. Then $\slam$ is
subnormal provided that one of the following two
conditions holds\/{\em :}
   \begin{enumerate}
   \item[(i)] $\tcal$ is rootless and
there exists a system $\{\mu_v\}_{v \in V}$ of Borel
probability measures on $\rbb_+$ which satisfies the
following equality for every $u \in V$,
   \begin{align} \label{wz1+}
\mu_u(\sigma) = \sum_{v\in \dzi{u}} |\lambda_v|^2
\int_{\sigma} \frac 1 s \D \mu_v(s), \quad \sigma \in
\borel{\rbb_+},
   \end{align}
   \item[(ii)] $\tcal$ has a root and
there exist $\varepsilon \in \rbb_+$ and a system
$\{\mu_v\}_{v \in V}$ of Borel probability measures on
$\rbb_+$ which satisfy \eqref{wz1+} for every $u \in
V^\circ$, and
   \begin{align*}
\mu_{\koo}(\sigma) = \sum_{v\in \dzi{\koo}}
|\lambda_v|^2 \int_{\sigma} \frac 1 s \D \mu_v(s) +
\varepsilon \delta_0(\sigma), \quad \sigma \in
\borel{\rbb_+}.
   \end{align*}
   \end{enumerate}
   \end{thm}
   \subsection{\label{cfs}Quasi-analytic vectors}
Let $S$ be an operator in a complex Hilbert space
$\hh$. We say that a vector $f\in\dzn{S}$ is a {\em
quasi-analytic} vector of $S$ if
   \begin{align*}
\sum_{n=1}^\infty \frac{1}{\|S^n
f\|^{\nicefrac{1}{n}}} = \infty \quad
\text{(convention: $\frac{1}{0}=\infty$)}.
   \end{align*}
Denote by $\quasi{S}$ the set of all quasi-analytic
vectors. Note that (cf.\ \cite[Section 9]{StSz1})
   \begin{align} \label{quasiinv}
   S(\quasi{S}) \subseteq \quasi{S}.
   \end{align}
In general, $\quasi{S}$ is not a linear subspace of
$\hh$ even if $S$ is essentially selfadjoint (see
\cite{ru2}; see also \cite{ru1} for related matter).

We now show that the converse of the implication in
Proposition \ref{necess} holds for weighted shifts on
directed trees having sufficiently many quasi-analytic
vectors, and that within this class of operators
subnormality is completely characterized by the
existence of a consistent system of probability
measures.
   \begin{thm} \label{main-0}
   Let $\slam$ be a weighted shift on a directed tree
$\tcal$ with weights $\lambdab=\{\lambda_v\}_{v \in
V^\circ}$ such that $\escr \subseteq \quasi{\slam}$.
Then the following conditions are equivalent{\em :}
   \begin{enumerate}
   \item[(i)] $\slam$ is subnormal,
   \item[(ii)] $\{\|\slam^n e_u\|^2\}_{n=0}^\infty$ is
a Stieltjes moment sequence for every $u \in V$,
   \item[(iii)] there exist a system $\{\mu_v\}_{v \in V}$ of Borel
probability measures on $\rbb_+$ and a system
$\{\varepsilon_v\}_{v \in V}$ of nonnegative real
numbers that satisfy \eqref{muu+} for every $u
\in V$.
   \end{enumerate}
   \end{thm}
   \begin{proof}
(i)$\Rightarrow$(ii) Apply Proposition \ref{necess}.

(ii)$\Rightarrow$(iii) Fix $u \in V$ and set $t_n =
\|\slam^{n+1} e_u\|^2$ for $n \in \zbb_+$. By
\eqref{st+1}, the sequence $\{t_n\}_{n=0}^\infty$ is a
Stieltjes moment sequence. Since $e_u \in
\quasi{\slam}$, we infer from \eqref{quasiinv} that
$\slam e_u \in \quasi{\slam}$, or equivalently that
$\sum_{n=1}^\infty t_n^{-\nicefrac{1}{2n}} = \infty$.
Hence, by the Carleman criterion for determinacy of
Stieltjes moment sequences (cf.\ \cite[Theorem
1.11]{sh-tam}), the Stieltjes moment sequence
$\{t_n\}_{n=0}^\infty = \{\|\slam^{n+1}
e_u\|^2\}_{n=0}^\infty$ is determinate. Now applying
Lemma \ref{2necess+} yields (iii).

(iii)$\Rightarrow$(i) Employ Theorem \ref{main}.
   \end{proof}

Using \cite[Theorem 7]{StSz1}, one can prove a version
of Theorem \ref{main-0} in which the class of
quasi-analytic vectors is replaced by the class of
analytic ones. Since the former class is
larger\footnote{\;In general, the class of analytic
vectors of an operator $S$ is essentially smaller than
the class of quasi-analytic vectors of $S$ even for
essentially selfadjoint operators $S$ (cf.\
\cite{ru1}).}, we see that ``analytic'' version of
Theorem \ref{main-0} is weaker than Theorem
\ref{main-0} itself. To the best of our knowledge,
Theorem \ref{main-0} is the first result of this kind;
it shows that the unbounded version of Lambert's
characterization of subnormality happens to be true
for operators that have sufficiently many
quasi-analytic vectors.

The following result, which is an immediate
consequence of Theorem \ref{main-0}, provides a
new characterization of subnormality of bounded
weighted shifts on directed trees written in
terms of consistent systems of probability
measures. It may be thought of as a complement to
Theorem \ref{charsub}.
   \begin{cor}
Let $\slam \in \ogr{\ell^2(V)}$ be a weighted
shift on a directed tree $\tcal$ with weights
$\lambdab=\{\lambda_v\}_{v \in V^\circ}$. Then
$\slam$ is subnormal if and only if there exist a
system $\{\mu_v\}_{v \in V}$ of Borel probability
measures on $\rbb_+$ and a system
$\{\varepsilon_v\}_{v \in V}$ of nonnegative real
numbers that satisfy \eqref{muu+} for every $u
\in V$.
   \end{cor}
   \subsection*{Acknowledgements} A substantial part of
this paper was written while the second and the fourth
authors visited Kyungpook National University during
the autumn of 2010 and the spring of 2011. They wish
to thank the faculty and the administration of this
unit for their warm hospitality.
   \bibliographystyle{amsalpha}
   
   \end{document}